\documentclass{article}
\usepackage{cmap}
\usepackage{amsthm}
\usepackage[cp1251]{inputenc}
\usepackage[T2A]{fontenc}
\usepackage[english]{babel}
\usepackage{amsfonts}
\usepackage{amssymb}
\usepackage{amsmath}
\usepackage{color}
\usepackage{hyperref}
\usepackage{mathdots}
\usepackage{cite}
\usepackage[numbers,sort&compress]{natbib}
\usepackage[left=4cm,right=4cm, top=5cm,bottom=5cm]{geometry}
\newtheorem{lemma}{Lemma}
\newtheorem{prop}{Proposition}
\newtheorem*{prop*}{Proposition}
\newtheorem{coroll}{Corollary}
\newtheorem{theorem}{Theorem}
\newtheorem*{theorem*}{Theorem}
\newtheorem*{Shmelkintheorem*}{Shmel'kin's theorem}
\theoremstyle{definition}
\newtheorem*{defin}{Definition}
\newtheorem{question}{Question}
\newtheorem*{example}{Example}
\theoremstyle{remark}
\newtheorem*{rem}{Remark}
\let\tilde\widetilde

\begin{document}

\title{On a generalization of Shmel'kin's theorem\footnote{
This work was supported by
the Russian Science Foundation, 
project no. 22-11-00075}}
\author{Mikhail A. Mikheenko
\\{\small
Faculty of Mechanics and Mathematics
of Lomonosov Moscow State University
}
\\
{\normalsize mamikheenko@mail.ru}
}
\date{}
\maketitle

\begin{abstract}
It is known that every nilpotent group
contains solution of every finite unimodular system of
equatiuons over itself. This statement, however, is not true for infinite systems.
Moreover, there are abelian groups which disprove the
infinite system analogue of the statement.
It has already been researched which periodic abelian groups
contain solutions of all infinite unimodular systems
of equations over themselves.
The present article covers the same question for periodic nilpotent groups
and for torsion-free nilpotent groups.
\end{abstract}

\section{Introduction}

Let us adopt the following convention on commutators:
$[g,h] = g^{-1} h^{-1} g h$.
Here we denote the field of order $p$ by $\mathbb Z_p$.
Let us also denote by $G^n$ the set $\left\{ x^n \mid x \in G \right\}$.
Denote the cyclic group of order $n$ with the generator
$g$ by $\langle g \rangle_n$.
The term ``neighbourhood'' here means open neighbourhood.

This article is devoted to solvability of equations over groups.

\begin{defin}
Let $G$ be a group.
An equation in variables $x_1,\ldots,x_n$
over $G$ is an expression (which may have coefficients from $G$)
$w(x_1,\ldots,x_n)=1$, where $w$ is an element of the free product $G*F(x_1,\ldots,x_n)$,
and $F(x_1,\ldots,x_n)$ is a free group with basis $x_1,\ldots,x_n$.

An equation $w(x_1,\ldots,x_n)=1$ is {\it solvable in the group}
$\tilde G$ if $G \subset \tilde G$
and $\tilde G$ contains a solution to said equation (i.e. there are elements
$\tilde g_1,\ldots,\tilde g_n \in \tilde G$ such that $w(\tilde g_1,\ldots, \tilde g_n)=1$).
This is the same as having a homomorphism $G*F(x_1,\ldots,x_n) \to \tilde G$,
which is injective on $G$ and maps $w$ to $1$.
In this case the group $\tilde G$ is called the {\it solution group}.

If an equation $w=1$ is solvable in some group $\tilde G$,
the equation $w=1$ over $G$ is called {\it solvable}.

Solvability of a (finite or infinite) system of equations 
(possibly in an infinite set of variables) is defined likewise.
\end{defin}

The study of equations over groups has a long history
yet remains an object of interest even now, see, for example,
\cite{B80, B84, K06, KP95, M24, Sh67, ABA21, BE18, EH91, EH21, EdJu00, 
GR62, How81, IK00, K93, KM23, KMR24, KT17, G83, Le62, NT22, P08, Sh81, T18, Kr85}.
There is also a survey on equations over groups \cite{Ro12}.

A lot of positive results on solvability
is known for nonsingular systems of equations over groups.

\begin{defin}
Let $\{w_i=1\}_{i \in I}$ be a system of equations
in a set of variables $\{x_j\}_{k \in J}$ over a group $G$.
For each equation $w_i$ consider a row
whose $j$th element is the exponent sum of $x_j$ in $w_i$.
For instance, the exponent sum of $y$
in an equation $x^2 y^{-3}g_1 xy^2 x y^2 x^{-1} y^{-2} g_2 = 1$
equals $-3 + 2 + 2 -2 = -1$.
The system of equations is called:
\begin{itemize}
\item {\it nonsingular} if said rows are linearly independent over $\mathbb Z$;
\item {\it $p$-nonsingular} for a prime number $p$ if the rows are linearly independent over $\mathbb Z_p$;
\item {\it $\pi$-nonsingular} for a set of (not necessarily all)
prime numbers if the system is nonsingular and $p$-nonsingular for every prime $p \in \pi$;
\item {\it unimodular} if the system is $p$-nonsingular for each prime number $p$.
\end{itemize}
\end{defin}
Note that by definition $\varnothing$-nonsingular systems are the same as
nonsingular systems.

One of the classical results on solvability of nonsingular systems
of equations is connected with finite groups.
\begin{theorem*}[\cite{GR62}]
Any nonsingular system of equations over a finite group
is solvable over this group.
\end{theorem*}

Another result on this topic is
concerned with locally indicable groups.
Recall that a {\it locally indicable} group is a group whose each nontrivial
finitely generated subgroup admits a surjective homomorphism onto $\mathbb Z$.
Note that this class of groups is, in some sense, an ``opposite''
of the class of finite groups.

\begin{theorem*}[\cite{How81}]
Any nonsingular system of equations over a locally indicable
group is solvable over this group.
\end{theorem*}

The last theorem has also an analogue for prime numbers.
Recall that a {\it locally $p$-indicable} group is a group whose
each nontrivial
finitely generated subgroup admits a surjective homomorphism onto $\mathbb Z_p$.

\begin{theorem*}[\cite{G83}, see also \cite{Kr85}]
Let $p$ be a prime number.
Then any $p$-nonsingular system of equations over a locally $p$-indicable
group is solvable over this group.
\end{theorem*}

Is it still remains unknown if there is a nonsingular system of equations over
some group which is not solvable.
Howie conjecture \cite{How81}, which states
that any nonsingular system of equations over any group is solvable,
still has yet to be either proved or disproved.

Apart from searching solutions of a system of equations in
any overgroup, one can try to find solutions in groups
similar to the initial group (in the best case, to find a solution
in the initial group itself).
The present paper is concerned with finding solutions of
nonsingular systems of equations in nilpotent groups.

The question is well-researched in the case of finite systems.

Recall that for a set of (not necessarily all) prime numbers $\pi$ a group $G$ is called
\begin{itemize}
\item {\it $\pi$-divisible} if for each prime $p \in \pi$ and each element $g \in G$
there is an element $h \in G$ such that $h^p = g$;
in particular any group is $\varnothing$-divisible;
\item {\it $\pi$-torsion free} if for each prime $p \in \pi$
the group $G$ has no elements of order $p$;
in particular, any group is a $\varnothing$-torsion free group.
\end{itemize}
Note that a divisible group is a $\Pi$-divisible groups,
while a torsion-free group is a $\Pi$-torsion free group,
where $\Pi$ is the set of all prime numbers.

There is a following theorem.
\begin{Shmelkintheorem*}[\cite{Sh67}]
Suppose that $G$ is a $\pi$-divisible
locally nilpotent
$\pi$-torsion free
group.

Then every finite
$\pi'$-nonsingular
system of equations over
$G$ is solvable in $G$.
Moreover, if there are as many equations in the system
as there are variables, then the solution is unique.
\end{Shmelkintheorem*}

Particularly, a divisible torsion-free nilpotent group contains
a solution of any finite nonsingular system of equations.
For finite unimodular systems the situation is even better,
as every nilpotent group contains a solution of any finite
unimodular system of equations over itself.

The last statement is not true for infinite systems:
abelian groups
$$
\mathbb Z_2 \oplus \mathbb Z_3 \oplus \mathbb Z_5 \oplus \ldots,
$$
$$
\mathbb Z_2 \oplus \mathbb Z_4 \oplus \mathbb Z_8 \oplus \ldots
$$
fail to have solutions to some infinite unimodular systems over themselves.
In \cite{M25} it was studied which periodic abelian groups
contain solutions of all infinite unimodular systems of equations over themselves.
\begin{theorem*}[\cite{M25}]
Suppose that $A$ is a periodic abelian group.
Then $A$ contains a solution of any unimoldular
system of equations over itself
if and only if
the reduced part of $A$ has bounded period.
\end{theorem*}

Some positive results for nilpotent groups,
partially extending Shmel'kin's theorem to
infinite systems,
were also achieved.
\begin{theorem*}[\cite{M25}]
Suppose that $G$ is a nilpotent group of bounded period.
Then any (not necessarily finite) unimodular system
of equations over $G$ is solvable in $G$ itself.
\end{theorem*}
\begin{theorem*}[\cite{M25}]
Suppose that $G$ is a divisible nilpotent group.
Then any (not necessarily finite) nonsingular
system of equations over $G$
has a solution in $G$ itself.
\end{theorem*}

However, the question remained unanswered for
torsion-free abelian groups,
as well as it remained unclear how to
extend the criterion from abelian groups
to nilpotent ones.

The following results provide answers to these questions.
\begin{theorem}\label{TheoremOne}
Suppose that $G$ is a periodic nilpotent group.

Then $G$ contains a solution of any unimodular system of
equations over itself
if and only if
only a finite number of $p$-components of
$G$ has non-trivial first Ulm factor and
all these factors have bounded period.
\end{theorem}

\begin{theorem}\label{TheoremTwo}
The following conditions on a nilpotent
torsion-free group $G$ are equivalent:
\begin{itemize}
\item $G$ contains a solution of any unimodular system of equations over itself.
\item $G$ is Raikov complete with respect to the group topology
in which the basic neighbourhoods of $1$ are $G^{n!}$.
\end{itemize}
\end{theorem}

Note that such groups topology is first-countable, hence
the completeness of $G$ can be viewed in pseudometric sense, that is
having a limit of any Cauchy sequence, Cauchy sequence being in that case
a sequence which stabilizes modulo $G^{n!}$ for any $n$.

Particularly, one can derive the following fact from Theorem \ref{TheoremTwo}.
\begin{coroll}\label{CountableTF}
A countable torsion-free nilpotent group $G$
contains a solution of any unimodular system of equations over itself
if and only if
it is divisible.
\end{coroll}

In Section \ref{SectionDiv} we prove some auxiliary statements
on the connection between divisible subgroups of a nilpotent group $G$
and such subgroups of $G/Z(G)$.

In Section \ref{SectionPer} we prove Theorem \ref{TheoremOne}.

In Section \ref{SectionProBurn} we introduce the group topology
mentioned in Theorem \ref{TheoremTwo}
(which we later call the pro-Burnside topology)
and investigate some of its properties.

In Section \ref{SectionTorsionFree} we prove Theorem \ref{TheoremTwo}
and Corollary \ref{CountableTF}.

In Section \ref{SectionQuestions} we formulate some open questions
which have yet to be answered.

The author thanks the Theoretical Physics and Mathematics
Advancement Foundation ``BASIS''.
The author also thanks A.Yu. Olshanskii who 
suggested a more simple proof for Proposition \ref{NontrivDivis}.

\setcounter{theorem}{0}

\section{Divisible subgroups of nilpotent groups}\label{SectionDiv}

Recall that the upper central series of a group $G$
is a normal series
$$
Z_0(G) \triangleleft
Z_1(G) \triangleleft
Z_2(G) \triangleleft
Z_3(G) \triangleleft
\ldots
$$
where $Z_0(G) = \{ 1 \}$ and
$Z_i(G)$ is a set of elements $h \in G$
such that $\left[h, g \right] \in Z_{i-1}(G)$ for any $g \in G$.
In other words, $Z_i(G)\subset G$ is a preimage of the center of the quotient group
$G/Z_{i-1}(G)$ under taking the quotient $G \to G/Z_{i-1}(G)$.

\begin{prop}\label{NontrivDivis}
Suppose that $G$ is such nilpotent group
that $G/Z(G)$ contains a non-trivial divisible subgroup.
Then $G$ contains a non-trivial divisible subgroup as well.
\end{prop}

\begin{proof} 
Suppose that $G/Z(G)$ contains a non-trivial divisible subgroup $\bar H$.
Consider its preimage $H \subset G$
under taking the quotient $G \to G/Z(G)$.
The quotient group is
$H / Z(G)$ divisible and non-trivial.

Take the largest $k$ such that
$\left[  H, Z_k(G) \right] = \{1 \}$. Note that $k$
is less than the nilpotency class of $G$
(as else $H$ satisfies the equation $[H,G] = \{ 1 \}$, that is
$H$ is central in $G$).
Then $\left[  H, Z_{k+1}(G) \right] \neq \{1 \}$.
In particular, $[h,c] \neq 1$
for some $h \in H, c \in Z_{k+1}(G)$.
Then consider the map
$h \mapsto [h,c]$
from $H$ to $G$.
Note that it is a group homomorphism. Indeed,
$$
[h_1h_2,c] = [h_1,c] \left[ [h_1,c],h_2 \right] [h_2,c]
$$
by a commutator identity. On the other hand, as
$c \in Z_{k+1}(G)$,
$[h_1,c] \in Z_k(G)$, particularly, this element commutes with $h_2$,
therefore
$$
[h_1h_2,c] = [h_1,c] [h_2,c].
$$
Also note that
$[h_1z,c] = [h_1,c]$ for $z \in Z(G)$.
In other words, the map $h \mapsto [h,z]$ induces a group homomorphism
$$
\varphi \colon H/Z(G) \to [H,Z_{k+1}(G)];
\quad
\varphi
\left(
hZ(G)
\right)
=
[h,c].
$$
As there is such element $h \in H$ that $[h,c] \neq 1$,
the image of $\varphi$ is non-trivial.
Since $H/Z(G)$ is a divisible group, the image of $\varphi$ is divisible as well.

So the image of $\varphi$ is the needed non-trivial divisible subgroup of $G$.

\end{proof}

\begin{lemma}\label{reducedfactor}
Suppose that $G$ is a reduced nilpotent group.
Then $G/Z(G)$ is reduced as well.
\end{lemma}

\begin{proof}
If $G/Z(G)$ contains a nontrivial divisible subgroup,
then $G$ contains a nontrivial divisible subgroup as well,
which is forbidden by hypothesis.

\end{proof}

\section{Periodic nilpotent groups}\label{SectionPer}

In this section we consider nilpotent $p$-groups.
Recall the definition of the height of an element
in such group.

\begin{defin}
Let $G$ be a nilpotent $p$-group.
The {\it height} of an element $g \in G$ is
the greatest number $n \in \mathbb N$ such that
$g = h^{p^n}$ for some $h \in G$.
If there is no greatest number among these, it is said that
$g$ has {\it infinite height}.
\end{defin}

\begin{prop}
Suppose that $G$ is a nilpotent $p$-group.
Then the set $G_0$ of elements of infinite height is
a normal subgroup such that
the quotient group $G/G_0$ has no elements of infinite height.
\end{prop}

\begin{proof}
The normality of the subset $G_0$ is obvious.
Now let $g,h \in G_0, k \in \mathbb N$.
Denote the nilpotency class of $G$ by $s$.
Then $g = r^{p^{k\cdot s}}, h = t^{p^{k\cdot s}}$. 
Therefore $gh = r^{p^{k\cdot s}}t^{p^{k\cdot s}}$.
Since for a nilpotent group $G$ of class $s$
it is true that
$\left\langle
G^{n^s}  
\right\rangle \subset G^n$,
we get that $gh \in G^{p^k}$ for any $k$.
So a product of two elements of $G_0$
has an infinite height as well, so
$G_0$ is a subgroup (as an inverse of an element
of infinite height has an infinite height as well).

Now suppose that $\bar a \in G/G_0$ is a nontrivial element of infinite height
in $G/G_0$.
Consider an element $a$ of the preimage of $\bar a$ under
taking the quotient $G \to G/G_0$.
For every $k$ there is an element $b\in G$
such that $b^{p^{k\cdot s}} = a c$, where $c \in G_0$.
For $c$ there is an element $d \in G$ such that
$d^{p^{k\cdot s}} = c^{-1}$.
Hence $a = b^{p^{k\cdot s}} d^{p^{k\cdot s}}$.
Again, we get that
$a \in \left\langle G^{p^{k\cdot s}} \right\rangle \subset G^{p^k}$
for any $k$, therefore $a \in G_0$.
We arrive at a contradiction with the non-triviality of $\bar a$.
That is why there are no nontrivial
elements of infinite height in $G/G_0$.

\end{proof}

\begin{defin}
The quotient group $G/G_0$
is called the {\it first Ulm factor}
of $G$.
\end{defin}

\begin{lemma} \label{fullhi}
Suppose that $G$ is a nilpotent $p$-group with
the first Ulm factor of bounded period.
Then $G_0$ is divisible.
\end{lemma}

\begin{proof}
Since the order of any element of $G_0$ is a power of
the prime number $p$,
it suffices to prove that for each $g \in G_0$
there is $t \in G_0$ such that $t^p = g$.

Let $g \in G_0$.
Suppose that the period of $G/G_0$ is $p^l$.
Then take an element $r \in G$,
such that $r^{p^{l+1}}=g$. As $r^{p^l} \in G_0$,
we conclude that for any element $g \in G_0$
there is an element $t \in G_0$ such that
$t^p = g$, as needed.

\end{proof}

\begin{lemma}\label{ulmbad}
Suppose that $G$ is a reduced nilpotent
$p$-group of unbounded period.
Then there is a unimodular system of equations over $G$
which has no solutions in $G$.
\end{lemma}

\begin{proof}
We argue by induction on the nilpotency class of $G$.

The base of induction (i.e. the case when $G$ is abelian) is proved in \cite{M25}.

Suppose that the statement of the lemma is proved
for nilpotent groups of class $s-1$.
Now we prove the statement for nilpotent groups
of class $s$.
Note that the quotient group $G/G_0$
is a nilpotent $p$-group with no elements
of infinite height which has unbounded period by Lemma
\ref{fullhi}, since $G_0$ is not divisible unless $G_0$
is trivial.
Hence $G$ can be assumed to have no elements of infinite height.

Now consider $G/Z(G)$. By Lemma
\ref{reducedfactor}, it is a reduced group.
If it has unbounded period, then by induction
there is a unimodular system of equations over $G/Z(G)$
which has no solutions in $G/Z(G)$.
Therefore, there is an analogous system of equations over $G$ as well.

Suppose, on the contrary, that $G/Z(G)$ has a bounded period $p^m$.
Then $Z(G)$ has unbounded period.
So $Z(G)$ is an abelian $p$-group of unbounded period with no
elements of infinite height.
Thus, it is separable, i.e.
for any $r \in \mathbb N$
there is a cyclic direct multiplier of $Z(G)$:
$$
Z(G) = \langle a \rangle_{\tilde r} \times B,
$$
such that $\tilde r > r$.
Hence there are following decompositions:
$$
Z(G) = \langle a_i \rangle_{k_i} \times B_i,
$$
such that $k_{i+1} - k_i > k_i$, i.e. $k_{i+1} > 2 k_i$,
and $k_1 > m$.

Consider then the following system of equations:
$$
\begin{cases}
x_1x_2^{-p^{k_1}}& = a_1, \\
x_2x_3^{-p^{k_2-k_1}} & = a_2, \\
x_3x_4^{-p^{k_3-k_2}}& = a_3,\\
& \vdots 
\end{cases}.
$$
Note that this system is unimodular, since the matrix of exponent sums of variables
in equations is as follows.
$$
\begin{pmatrix}
1 & -p^{k_1} & 0 & 0 & \ldots \\
0 & 1 & -p^{k_2-k_1} & 0 & \ldots \\
0 & 0 & 1 & -p^{k_3-k_2} & \ldots \\
0 & 0 & 0 & 1 & \ldots \\
\vdots & \vdots & \vdots& \vdots& \ddots
\end{pmatrix}
$$

Suppose that the group $G$ contains a solution $\{y_i\}_{i \in \mathbb N}$ of the system above.
In that case the elements $y_2^{-p^{k_1}}, y_3^{p^{k_2-k_1}}, y_4^{-p^{k_3-k_2}}$, and so on,
lie in the center of $G$, therefore, in particular, we get that
$$
a_2^{p^{k_1}} = \left( y_2y_3^{-p^{k_2-k_1}} \right)^{p^{k_1}} = y_2^{p^{k_1}}y_3^{-p^{k_2}},
$$
$$
a_3^{p^{k_2}} =\left( y_3y_4^{-p^{k_3-k_2}} \right)^{p^{k_2}} = y_3^{p^{k_2}} y_4^{-p^{k_3}},
$$
and so on. By multiplying these expressions we get
$$
y_1 \cdot y_{n+1}^{-p^{k_n}} = a_1 \cdot a_2^{p^{k_1}} \cdot \ldots \cdot a_n^{p^{k_{n-1}}}.
$$
Since the order of $a_i^{p_{i-1}}$ equals $p^{k_i - k_{i-1}}$,
the order of the right-hand side of this equation equals the order of
$a_n^{p^{k_{n-1}}}$ (as the other elements have smaller orders),
that is $p^{k_n - k_{n-1}}$.

Also note that $y_{n+1}^{p^{k_n}} = z^{p^{k_n - m}}$
for some $z\in Z(G)$.
In the decomposition
$$
Z(G) = \langle a_n \rangle_{p^{k_n}} \times B_n
$$
the first component of $z^{p_{k_n} - m}$
has order at most $p^m$.
Therefore, if $k_n - k_{n-1}>m$, then
the order of $y_1$ is at least $p^{k_n-k_{n-1}}$,
as the order of $y_{n+1}^{p^{k_n}} = z^{p^{k_n - m}}$
in that case is less than $p^{k_n - k_{n-1}}$.

Since
$$
k_i - k_{i-1} > k_{i-1} - k_{i-2} > \ldots > k_2 - k_1 > k_1 > m,
$$
the condition
$$
k_i - k_{i-1} > m
$$
is indeed satisfied for all $i \in \mathbb N$.
On the other hand,
$$
k_i - k_{i-1} > k_{i-1} > k_{i-2} \cdot 2> \ldots > k_1 \cdot 2^{i-2} > m \cdot 2^{i-2}.
$$
That is why the order of $y_1$ is at least $p^{m \cdot 2^{n-2}}$
for any sufficiently big number $n$.
In other words, $y_1$ has infinite order, which is impossible
in the periodic group $G$. Thus, the system above has no solutions in $G$.

As a conclusion, a unimodular system of equations over $G$,
which has no solutions in $G$ itself,
is found both in case when $G/Z(G)$ has bounded period
and in case when the period of $G/Z(G)$ is unbounded,
hence the proof of the lemma is complete.

\end{proof}

\begin{lemma} \label{ulmgood}
Suppose that $G$ is a nilpotent
$p$-group with the first Ulm
factor of bounded period.
Then every $p$-nonsingular
system of equations over $G$
has a solution in $G$.
\end{lemma}

\begin{proof}
Note that the subgroup $G_0\subset G$ is divisible by Lemma \ref{fullhi}.
We argue by induction on the nilpotency class of $G$.
For abelian groups we have a decomposition of $G$
into a direct product of $G_0$
and the first Ulm factor of $G$, which has bounded period.
Such groups contain solutions of all $p$-nonsingular systems of equations over themselves
\cite{M25}.

Suppose that the statement of the lemma is proved for
nilpotent groups of nilpotency class $s-1$.
Now we prove the statement for nilpotent groups of class $s$.
Consider $G/\gamma_{s}(G)$. As $(G/\gamma_{s}(G))_0$
contains the image of $G_0$ under taking the quotient
$G \to G/\gamma_{s}(G)$,
the first Ulm factor of
$G/\gamma_{s}(G)$ has bounded period as well.
Hence, by induction, $G/\gamma_{s}(G)$ contains a solution of any
$p$-nonsingular system of equations over itself.
This means that any $p$-nonsingular system of equations over $G$,
after a fitting substitution of variables,
can be made such that
the product of coefficients of any equation lies in $\gamma_{s}(G)$.
So from now we consider only systems with this property.

Assume that the variables in the system commute with each other as well as with the coefficients of $W$.
In that case the system $W$ is equivalent to a system $W'$, in which every equation
has no coefficients apart from one coefficient lying in $\gamma_{s}(G)$.
Note that a set of elements from $\gamma_{l+1}(G)$ is a solution of $W$
if and only if this set is a solution of $W'$.
Indeed, the subgroup $\gamma_{l+1}(G)$
is central in $G$, therefore the elements of such solution
actually commute with each other and with coefficients of $W$,
so in this case the systems $W$ and $W'$ are indeed equivalent.
As the system $W$ is $p$-nonsingular,
such is the system $W'$ as well. Thus,
to prove the lemma, it suffices to prove that $\gamma_{l+1}(G)$
satisfies the conditions of the lemma.

We prove that there is a number $t$
such that $\left( \gamma_{s}(G) \right)^t$ is divisible
(from this fact follows that the reduced part of
$\gamma_{s}(G)$ has bounded period, as needed).
Take $[g,h]$, where $g \in G, h \in \gamma_{s-1}(G)$.
Suppose that the period of the first Ulm factor of $G$
equals $m$. Then $g^m \in G_0$. So for any natural number $k$
there is an element $r \in G$ such that $g = r^k$.
It follows that
$$
[g^m, h] = [r^{km},h] = [r,h]^{km}
\mod{\gamma_{s+1} = \{1\}}.
$$
Also note that $[g^m,h] = [g,h]^m \mod{\gamma_{s+1}(G) = \{1\}}$.
As a conclusion,
$$
[g,h]^m
=
[r,h]^{km}.
$$
In other words, for any element $c$ of the form $[g,h]^m \left( \gamma_{s}(G) \right)^m$,
where $g \in G, h \in \gamma_{s-1}(G)$,
and any integer $k$ there is an element $b:=[r,h]^m \in \left( \gamma_{s}(G) \right)^m$ 
such that $c = b^k$.
As these elements generate the abelian subgroup
$\left( \gamma_{s}(G) \right)^m$, we conclude that this subgroup is divisible.

\end{proof}

The proved lemmas allow us to prove Theorem
\ref{TheoremOne}, which was formulated in the introduction.
Recall its statement.

\begin{theorem}
Suppose that $G$ is a periodic nilpotent group.

Then $G$ contains a solution of any unimodular system of
equations over itself
if and only if
only a finite number of $p$-components of
$G$ has non-trivial first Ulm factor and
all these factors have bounded period.
\end{theorem}

\begin{proof}
Denote the $p$-component of $G$ by $G_p$.

If there is component $G_p$
of $G$ with unbounded period of its first Ulm factor,
then, by Lemma \ref{ulmbad}, there is a unimodular
system over $G_p$ with no solutions in $G_p$
(and, consequently, there is an analogous system over $G$).

If, nonetheless, all of the first Ulm factors of $p$-components have bounded periods,
but there is an infinite number of nontrivial ones among them,
then $G$ contains no solutions of the system
$$
\begin{cases}
x y_1^{-p_1^{m_1}}& = a_1,\\
x y_2^{-p_2^{m_2}}& = a_2 , \\
& \vdots 
\end{cases}
$$
where $p_i^{m_i}$ is a period of $G_{p_i}/\left( G_{p_i} \right)_0$,
while $a_i \in G_{p_i} \setminus (G_{p_i})_0 $.
Indeed, having a solution of this system assumes that the element $x$
of this solution has a nontrivial $p_i$-component for every $i$
(since the corresponding component of  $y_i^{p_i^{m_i} }$
lies in $(G_{p_i})_0 \not\ni a_i$),
thus $x$ has infinite order, which is impossible in a periodic group.

Now note that the matrix of exponent sums of variables in the equations of the system is
$$
\begin{pmatrix}
1 & -p_1^{m_1} & 0 & 0 & 0 & \ldots \\
1 & 0 & -p_2^{m_2} & 0 & 0 & \ldots \\
1 & 0 & 0 & -p_3^{m_3} & 0& \ldots \\
1 & 0 & 0 & 0  & -p_4^{m_4}& \ldots \\
\vdots & \vdots & \vdots& \vdots & \vdots & \ddots
\end{pmatrix}.
$$
It follows that the system above is $p$-nonsingular
both for prime numbers $p_i$, which are included in this matrix,
(since a power of each of such prime numbers is included in the matrix only once)
and for the remaining prime numbers $p$.
In other words, this system, which has no solutions in $G$, is unimodular.

If, on the contrary, there is only a finite number of $p$-components $G_p$
of $G$ with nontrivial first Ulm factors, and all these factors have bounded period,
then $G$ is decomposed into a direct product of
a divisible nilpotent group
and a finite number of
groups, which satisfy Lemma \ref{ulmgood}.
Such product contains a solution to every
unimodular system of equations over itself.

\end{proof}

\section{Pro-Burnside topology}\label{SectionProBurn}

In this and next sections we use additive notation for abelian groups.
For an abelian group $A$ consider a group topology with basic neighbourhoods of $0$ being $nA$.
Note that quotient groups by such neighbourhoods have bounded periods,
that is why we call this topology {\it pro-Burnside},
analogously to the profinite topology,
or {\it factor-adic} (analogously to the $p$-adic topology).

Neighbourhoods of the profinite topology are open in the pro-Burnside topology as well,
since finite groups have bounded period.
The converse, however, is not always true, as
the neihgbourhood $2A$ of the infinite direct sum
$$
A = \mathbb Z \oplus \mathbb Z \oplus \ldots
$$
is not open in the profinite topology (as the respective quotient group is not finite).
So the pro-Burnside topology is generally
finer than the profinite topology.

\begin{prop}
Both families $\{nA \mid a \in A, n \in \mathbb N\}$ and
$\{n! A \mid a \in A, n \in \mathbb N\}$
can be taken to be a neighbourhood
base of $0$ in the pro-Burnside topology on the abelian group $A$.
\end{prop}
\begin{proof}
The family $\{nA \mid n \in \mathbb N\}$ is a neighbourhood base of $0$ in the pro-Buenside topology by definition.
Now we prove that the family $\{n!A \mid n \in \mathbb N\}$ is a neighbourhood base of $0$ in this topology as well.
Indeed, on the one hand, $n!A = kA$ for $k=n!$, thus
the subsets $n!A$ are open in the pro-Burnside topology.
On the other hand, as $n!A \subset nA$, $0$ is contained in any open subset
with some neighbourhood of the family $\{n!A\mid n \in \mathbb N\}$. 
\end{proof}

Note that the subsets $n!A$ are included one into another:
$$
A \supset 2A \supset 3!A \supset 4!A \supset \ldots.
$$

Likewise, we consider a group topology on a nilpotent group $G$
the neighbourhoods of $1$ in which are $G^n$. We call such topology
{\it pro-Burnside} or {\it factor-adic} as well.
As previously mentioned, we can also choose the family $\{G^{n!}\}$
to be the neighbourhood base of $1$ in this topology. Let us however note that the subsets $G^n$ 
are not subgroups. But this topology can still be viewed as a linear topology
(that is, a topology which admits a neighbourhood base of $0$ consisting of subgroups)
by virtue of the following proposition.
\begin{prop}
Suppose that $G$ is a nilpotent group of class $s$.
Then the family $\left\langle G^{n^s} \right\rangle$
can be taken to be the neighbourhood base of $1$
in the pro-Burnside topology.
\end{prop}
\begin{proof}
The following inclusions hold in the nilpotent group $G$ of class $s$:
$$
G^{n^s} \subset \left\langle G^{n^s} \right\rangle \subset G^n.
$$
\end{proof}

\begin{rem}
We can also choose the neighbourhood base of $1$ to consist of subgroups of the form
$\left\langle G^{n} \right\rangle$
or to consist of subgroups of the form
$\left\langle G^{n!} \right\rangle$.
\end{rem}

Let us also stress the case, in which the pro-Burnside topology
on a group is Hausdorff.

\begin{prop}
The pro-Burnside topology on a reduced nilpotent torsion-free group
is Hausdorff.
\end{prop}
\begin{proof}
Suppose that $G$ is such group.
Then $\bigcap\limits_{n \in \mathbb N}G^{n!} = \{1\}$, as else there exist such elements $1 \neq g_1, g_2, g_3, \ldots\in G$ that
$$
g_1 = g_2^2 = g_3^{3!} = g_4^{4!} = \ldots.
$$
Since taking an $n$th root of an element of a nilpotent torsion-free group are unique,
it follows that $g_2 = g_3^3$, $g_3 = g_4^4$ and so on.
Such elements generate a non-trivial divisible subgroup of $G$,
which is forbidden by hypothesis.
Thus, the pro-Burnside topology on $G$ is Hausdorff.

\end{proof}

Recall that a topological group $G$ is called Raikov complete if any
Cauchy filter on $G$ converges to some element of $G$,
where a Cauchy filter on a topological group $G$ is such filter $\eta$
that for any neighbourhood $V\subset G$ of unity there are such elements
$a,b \in G$ that $aV\in \eta$ and $Vb \in \eta$ \cite{AT08}.
Since the pro-Burnside topology on a group is first-countable,
hence is generated by some invariant prenorm \cite{AT08},
Raikov completeness in this case
can be formulated in more simple terms.

\begin{prop}\label{ProBurn}
The following conditions on a nilpotent group $G$ are equivalent:
\begin{enumerate}
\item $G$ is Raikov complete in the pro-Burnside topology.
\item $G$ is complete with respect to any invariant prenorm, which generates the pro-Burnside topology.
\item Any sequence $g_1,g_2,\ldots \in G$,
which stabilizes modulo
$ \left\langle
G^{n!}
\right\rangle$
for any natural number $n$,
has a limit in the pro-Burnside topology.
\end{enumerate}
\end{prop}
\begin{proof}

Suppose that $\mathfrak n$ is an invariant prenorm on $G$ which generates the pro-Burnside topology.
Let $B_n = \left\{ h \in G \mid \mathfrak n(h) < \dfrac{1}{n}\right\}$,
$V_n = \left\langle
G^{n!}
\right\rangle$.

Firstly, prove that Cauchy sequences with respect to the prenorm $\mathfrak n$
and sequences, which stabilize modulo any $V_n$,
are the same sequences. %
For this we note that a sequence $g_1, g_2, \ldots \in G$
is Cauchy with respect to the prenorm $\mathfrak n$
if and only if
for any $n\in \mathbb N$
there is such number $N\in \mathbb N$ that
$g_{n_1} \in g_{n_2} B_n$ if $n_1, n_2 > N$.
Meanwhile, the same sequence stabilizes modulo $V_n$, %
if there is such number $N\in \mathbb N$ that for $n_1, n_2 > N$
we have $g_1, g_2 \in g_N V_n$. As $V_n$  is a subgroup, it follows that
$g_1 \in g_2 V_n$.
It remains to note that for any neighbourhood $B_n$ there is a neighbbourhood $V_{n'} \subset B_n$ and
for any neighbourhood $V_n$ there is a neighbourhood $B_{n'} \subset V_n$.

Now let us investigate Cauchy filters in the pro-Burnside topology.
To prove that the filter is Cauchy, one only has to check the condition for
basic neighbourhoods of unity, that is for subgroups $V_n$ (if we take the corresponding neighbourhood base).
Since $gV_n = V_ng$, the filter $\eta$ is Cauchy
if and only if
for any $n\in \mathbb N$
there is an element $g_n\in G$ such that $g_nV_n \in \eta$.
In particular, for any natural numbers $n<m$
$g_nV_n\cap g_mV_m \neq \varnothing$. Since $V_m \subset V_n$, whereas $V_n$ is a subgroup,
it follows that $g_m \in g_nV_n$ for $m>n$.
Hence the sequence $g_1, g_2, g_3, \ldots$ stabilizes modulo
$V_n$ for each $V_n$ (and, equivalently, is Cauchy with respect to $\mathfrak n$).
At the same time the convergence of the filter $\eta$ to an element $g$ means that
for any $n$ $gV_n\in \eta$, hence, repeating the argument above, we get that for $m>n$ $g_m \in g V_n$.
In other words, $g_n \to g$.

So, if any Cauchy sequence converges to some element, the same is true
for Cauchy filters.
The converse is also true, as from a Cauchy sequence $g_1, g_2, \ldots \in G$
one can construct a filter $\eta$,
generated by neighbourhoods $g_1V_1, g_2V_2, \ldots$,
which is a Cauchy filter.
If the filter $\eta$ converges to $g$, then (again, by virtue of the argument above)
$g_n \to g$, which finishes the proof.

\end{proof}

Let us call a nilpotent group $G$
{\it pro-Burnside} if it satisfies any condition from Proposition
\ref{ProBurn}.
For torsion-free nilpotent groups
the property of being pro-Burnside is inherited by
the center and the quotient group by the center.

\begin{lemma} \label{probounded}
Suppose that torsion-free nilpotent group $G$
is pro-Burnside.
Then $Z(G)$ and $G/Z(G)$ are pro-Burnside as well.
\end{lemma}

\begin{proof}
Note that a torsion-free nilpotent group satisfies the condition
$Z(G) \cap G^n = Z(G)^n$,
hence the topology on $Z(G)$, induced by the pro-Burnside topology on
$G$, coincides with the pro-Burnside topology on $Z(G)$.
Since $Z(G)$ is a closed subgroup of a complete pseudometric group $G$
(indeed, $Z(G)$ is an intersection of closed subsets $\{x \in G \mid [g,x] = 1\}$
for all elements $g \in G$),
$Z(G)$ is complete in the induced topology, which coincides with the pro-Burnside topology of $Z(G)$.
This means that $Z(G)$ is a pro-Burnside group.

Now prove that the quotient group $G/Z(G)$ is pro-Burnside.
Suppose that a sequence $\bar g_1, \bar g_2 , \ldots \in G/Z(G)$ stabilizes modulo
$\left\langle
\left(
G/Z(G)
\right)^{n!}
\right\rangle$
for any natural number $n$.
By taking a subsequence, we can assume that
$$
\bar g_n = \bar g_{n+1} \mod \left\langle
\left(
G/Z(G)
\right)^{n!}
\right\rangle.
$$

Consider some preimages $g_n \in G$ of the elements $\bar g_n \in G/Z(G)$
under taking the quotient by $Z(G)$.
We get that for any $n$
$$
g_n =g_{n+1} z_n \mod \left\langle
G^{n!}
\right\rangle,
$$
where $z_n \in Z(G)$.
Now take the sequence
$$
h_1 = g_1, h_2 = g_2 z_1, h_3 = g_3 z_2 z_1, \ldots.
$$
Note that
$$
h_n h_{n+1}^{-1} = g_n g_{n+1}^{-1} z^{-1}_{n} = z_n z_n^{-1} = 1 \mod \left\langle
G^{n!}
\right\rangle.
$$
As $G$ is pro-Burnside, it follows that $h_n \to g$ for some element $g \in G$.
Moreover,
$$
h_n = g \mod \left\langle
G^{n!}
\right\rangle
$$
for any $n \in \mathbb N$, since
$$
h_n = h_{n+1} = h_{n+2} = \ldots \mod \left\langle
G^{n!}
\right\rangle.
$$

The equations also hold after taking the quotient by $Z(G)$:
$$
\bar h_n = \bar g \mod
\left\langle
\left(
G/Z(G)
\right)^{n!}
\right\rangle,
$$
where $\bar h_n, \bar g \in G/Z(G)$ are respective images of the elements
$h_n, g\in G$ under taking the quotient by $Z(G)$.
Therefore, $\bar h_n \to \bar g$ in the pro-Burnside topology of $G/Z(G)$.
It remains to notice that $\bar h_n = \bar g_n$, as
$h_n \in g_n Z(G)$. Hence, the sequence $\bar g_n \in G/Z(G)$
has a limit, as needed.

\end{proof}

\section{Torsion-free groups}\label{SectionTorsionFree}

A torsion-free nilpotent group being pro-Burnside
is equivalent to this group containing a solution of any
unimodular system of equations over itself.
First we show this for abelian groups.

\begin{lemma}\label{torsionfree}
Suppose that $A$ is a torsion-free abelian group.
Then $A$ contains a solution of any unimodular system of equations over itself
if and only if
$A$ is pro-Burnside.
\end{lemma}

\begin{proof}
The abelian group $A$ can be decomposed into a direct sum $A = B \oplus C$,
where $B$ is divisible and $C$ is reduced.
Since the divisible abelian group $B$ contains a solution of any nonsingular system of equations over $B$,
$A$ contains a solution of any unimodular system of equations over itself
if and only if
the same is true for $C$.
On the other hand, as $nB = B$ for any $n \in \mathbb N$,
one can see that
$A$ is complete in its pro-Burnside topology (viewed as a topology generated by a prenorm)
if and only if
$C$ is complete in the pro-Burnside topology of $C$.
As a conclusion, $A$ can be assumed to be reduced.
Particularly, the pro-Burnside topology on $A$
is Hausdorff, and the limit of any convergent sequence in $A$ is unique.
 
Suppose, firstly, that $A$ is pro-Burnside.
Consider a unimodular system of equations
$$
\left\{m_i = a_i \mid  i \in I\right\},
$$
where $m_i$ is a linear combination of variables $\{x_j\}$ with coefficients from $\mathbb Z$ and $a_i \in A$.
This system has a solution modulo $2A$, since the system is unimodular, while $A/2A$ is an abelian group of bounded period.
In other words, there is a set of elements $\{\alpha_{j2}\}$ such that
$$
\mu_{i2} = a_i \mod{2A},
$$
where $\mu_{i2}$ is $m_i$ after substituting $x_j$ with $\alpha_{j2}$.
So we have $a_{i2} : = a_i - \mu_{i2} \in 2A$.
Now consider the system
$$
\left\{ m_i = a_{i2} \right\}.
$$
As $a_{i2} \in 2A$, the new system can be viewed as a system of equations over $2A$.
It is unimodular as well, therefore it has a solution modulo $3 \cdot 2A = 3!A$.
That is, there are elements $\alpha_{j3} \in 2A$ such that
$$
\mu_{i3} = a_{i2} \mod{3!A}.
$$
Now note that is $x_j = \alpha_{j2} + \alpha_{j3}$, then
$$
m_i = \mu_{i2} + \mu_{i3} = a_{i} - a_{i2} + a_{i2} = a_i \mod{3!A}.
$$
Likewise, we get elements $\alpha_{j4} \in 3!\cdot A, \alpha_{j5} \in 4!\cdot A, \ldots, \alpha_{jk} \in (k-1)!\cdot A$
such that,
if $x_j = \alpha_{j2} + \ldots + \alpha_{jk}$, then
$$
m_i = a_i \mod {k! A}.
$$
Now take the series $\tilde x_j = \sum\limits_{k=2}^\infty \alpha_{jk}$,
which converge because $\alpha_{jk} \in (k-1)!A$ and $A$ is pro-Burnside.
After substituting $x_j$ in $m_i$ with $\tilde x_j$, we get that
$$
m_i = a_i \mod {k!A}
$$
for any natural number $k$.
Since pro-Burnside topology on $A$ is Hausdorff,
it means $m_i = a_i$, thus $\tilde x_j$ is a needed solution of the system.

Now suppose that $A$ is not pro-Burnside.
In this case there is a sequence
$a_1, a_2 , \ldots \in A$, which stabilizes modulo $n!A$ for every $n$
but has no limit.
Taking a subsequence, we can assume that
$$
a_n - a_{n+1} \in (n+1)!A.
$$
This is equivalent to the fact that the series
$\sum\limits_{n = 1}^\infty n!b_n$, where $n!b_n = a_{n-1} - a_{n} \in n!A$, does not converge.
Consider then the system of equations
$$
\begin{cases}
x_1 - 2x_2 = b_1, \\
x_2 - 3 x_3 = b_2, \\
x_3 - 4x_4 = b_3, \\
\ldots
\end{cases}
$$
This system implies $x_1 - (n+1)! x_{n+1} = b_1 + 2b_2 + 3!b_3 + \ldots + n!b_n$.
Assume that $A$ contains a solution $\{\tilde x_j\}$ of this system.
Then $\tilde x_1 = \sum\limits_{n=1}^\infty n!b_n$,
as $(n+1)! \tilde x_{n+1} \to 0$.
We arrive to a contradiction with the fact that the series $\sum\limits_{n = 1}^\infty n!b_n$ has no limit.
Therefore the system has no solutions in $A$.

\end{proof}

Now we can move on to nilpotent groups,
in other words, prove Theorem
\ref{TheoremTwo} stated in the introduction.
Here we state it in pro-Burnside terms.

\begin{theorem}[In pro-Burnside terms]
A torsion-free nilpotent group $G$
contains a solution of any unimodular system of equations over itself
if and only if
$G$ is pro-Burnside.
\end{theorem}

\begin{proof}
We argue by induction on the nilpotency class of $G$.

For nilpotent groups of class $1$ (in other words, for abelian groups) the statement is true by Lemma \ref{torsionfree}.

Now assume that the theorem is proved for nilpotent groups of class $s-1$.
We prove it for nilpotent groups of class $s$.

Firstly, assume that $G$ is pro-Burnside.
Then, by Lemma \ref{probounded},
the quotient group $G/Z(G)$ is pro-Burnside as well.
This quotient group is also torsion-free \cite{Mal49}.
By induction, $G/Z(G)$ contains a solution of any unimodular system of equations over itself.
For the group $G$ is means that any unimodular system of equations over $G$
can be transformed into a unimodular system in which every equations has the product of all its coefficients lying in $Z(G)$.
Assume that some solution of this system lies in $Z(G)$.
In this case we can assume that the variables commute with
coefficients of equations, hence all equations can be rewritten in the form $w_i = z_i$, where $w_i$ have no coefficients, while $z_i \in Z(G)$.
Thus, we get a unimodular system of equations over $Z(G)$.
It actually has a solution in $Z(G)$, as $Z(G)$ is also pro-Burnside by Lemma
\ref{probounded}.

Now suppose that $G$ is not pro-Burnside.
There are two possibilities:
\begin{itemize}
\item The quotient group $G/Z(G)$ is not pro-Burnside. Then, by induction hypothesis,
there is a unimodular system of equations over $G/Z(G)$ with no solutions in $G/Z(G)$.
Then there is an analogous system of equations over $G$.
\item The quotient group $G/Z(G)$ is pro-Burnside.
Then consider a Cauchy sequence $g_1, g_2, g_3, \ldots$ of elements of the group $G$
which has no limit.
Consider the images $\bar g_n \in G/Z(G)$ of the elements $g_n \in G$ under taking the quotient by $Z(G)$.
The sequence $\bar g_1, \bar g_2, \bar g_3, \ldots$ is Cauchy in $G/Z(G)$, hence it has a limit.
So, multiplying all the elements $\bar g_n$ by the same element, we can assume that
$\bar g_n \to 1$.
Taking a subsequence, we can assume that
$$\bar g_n = 1 \mod \left\langle
\left(
G/Z(G)
\right)
^{n!}
\right\rangle.
$$
For the sequence $g_n \in G$ it means that
$$
g_n = z_n \mod
\left\langle
G^{n!}
\right\rangle,
$$
where $z_n \in Z(G)$.
Since $g_nz_n^{-1} \to 1$,
the sequence $z_n \in Z(G)$ is Cauchy as well, but also has no limit.

Taking the subsequence of the sequence
$z_n \in Z(G)$, we can assume that
$$
z_n = z_{n+1} \mod G^{n!},
$$
that is  $z_n z_{n+1}^{-1} = b_n^{n!}$ for some element $b_n \in G$.
As the nilpotent group $G$ is torsion-free,
$G^{n!} \cap Z(G) = Z(G)^{n!}$, hence $b_n \in Z(G)$.
Note that
$$
b_1 b_2^2 b_3^{3!} \cdot \ldots \cdot b_n^{n!} = z_1 \cdot z_2^{-1} \cdot z_2 \cdot z_3^{-1} \cdot \ldots \cdot z_n \cdot z_{n+1}^{-1} = z_1 \cdot z_{n+1}^{-1}.
$$

Now consider the system of equations
$$
\begin{cases}
x_1 x_2^{-2} = b_1, \\
x_2x_3^{-3} = b_2, \\
x_3 x_4^{-4} = b_3, \\
\ldots
\end{cases}
$$
It is easy to see that this system is unimodular.
Assume that is has a solution $\{\tilde x_n\}$  $G$.
As $\tilde x_n = b_n \tilde x_{n+1}$, $\tilde x_n^{n!} = (b_n \tilde x_{n+1})^{n!} = \tilde x_{n+1}^{n!} b_n^{n!}$ because $b_n \in Z(G)$.
Therefore, this system implies
$$
\tilde x_1 (\tilde x_{n+1})^{-(n+1)!} = b_1 \cdot b_2^2 \cdot b^{3!} \cdot \ldots \cdot b_{n}^{n!}.
$$
Thus
$$
\tilde x_1 = z_1 z_{n+1}^{-1} (\tilde x_{n+1})^{-(n+1)!}.
$$
On one hand, $(\tilde x_{n+1})^{-(n+1)!} \to 1$. On the other hand, in this case $z_{n+1}\to\tilde x_1^{-1} z_1$,
although the sequence $z_n$ has no limit.
We arrive at a contradiction, so the assumption of having a solution to the system in
$G$ is false.
\end{itemize}
So, in both of these possibilities we have a unimodular system of equatons over $G$
which has no solutions in $G$, as needed.

\end{proof}

In the countable case groups,
which contain a solution to any unimodular system of equations,
admit a more simple description.

\begin{coroll}
A countable trosion-free nilpotent group $G$ без кручения
contains a solution of any unimodular system of equations over itself
if and only if
$G$ is divisible.
\end{coroll}

\begin{proof}
The case, in which $G$ is divisible, was studied in \cite{M25}.

Suppose now that $G$ is not divisible. By $s$ denote the nilpotency class of $G$.
Take the subgroup $H= \bigcap\limits_{n \in \mathbb N}\left\langle G^{n!} \right \rangle$.
As $\left\langle G^{(n!)^s!} \right\rangle \subset G^{n!}$, the subgroup $H$ is divisible and, in particular, not equal to $G$.
Suppose that $g^k \in H$. Then
$$
g^k \in \left\langle G^{(n!k)^s!} \right\rangle\subset G^{n!k},
$$
hence $g \in G^{n!}$ for any $n \in \mathbb N$, which implies $g \in H$.
Therefore, the quotient group $G/H$ is torsion-free.
The pro-Burnside topology on $G/H$ is metrizable
(since it is Hausdorff and first-countable), whereas
$G/H$ itself is non-trivial but at most countable.
By Baire category theorem, $G/H$
is not complete in its pro-Burnside topology,
hence, by Theorem \ref{TheoremTwo}, there is a unimodular system of equations over $G/H$
with no solutions in $G/H$.
Therefore, there is an analogous system of equations over $G$.

\end{proof}
 
We also provide an example, which shows that this statement is not true for non-countable groups.
\begin{example}
Take the additive group $\mathfrak Z_p$
of $p$-adic integer numbers.
The profinite and the pro-Burnside topologies on this group are the same,
as both admit a neighbourhood base of $0$ consisting of subgroups of the form
$n \mathfrak Z_p$
(some of these neighbourhoods may coincide).

It is known that $\mathfrak Z_p$ is complete in its profinite topology, hence it is pro-Burnside,
which means that any unimodular system of equations over $\mathfrak Z_p$ has a solution in $\mathfrak Z_p$.
However, $\mathfrak Z_p$ is not divisible, as $p\mathfrak Z_p \neq \mathfrak Z_p$.

\end{example}

\section{Open questions}\label{SectionQuestions}

The case, when the nilpotent group is neither periodic nor torsion-free,
remains unstudied. Such groups are called mixed.

\begin{question}
Which mixed abelian groups
contain solutions of all infinite unimodular systems of equations
over themselves?
\end{question}

\begin{question}
Which mixed nilpotent groups
contain solutions of all infinite unimodular systems of equations
over themselves?
\end{question}

\bibliography{biblioeng}

\end{document}